\documentclass[11pt]{amsart}
\usepackage{amssymb,mathrsfs,graphicx,enumerate}
\usepackage{colortbl}
\definecolor{black}{rgb}{0.0, 0.0, 0.0}
\definecolor{red}{rgb}{1.0, 0.5, 0.5}

\title[   ]{Dynamics of a spatially homogeneous Vicsek model for oriented particles on the plane}

\author[Kang]{Moon-Jin Kang}
\address[Moon-Jin Kang]{\newline Department of Mathematics, \newline The University of Texas at Austin, Austin, TX 78712, USA}
\email{moonjinkang@math.utexas.edu}

\author[Morales]{Javier Morales}
\address[Javier Morales]{\newline Department of Mathematics, \newline The University of Texas at Austin, Austin, TX 78712, USA}
\email{jmorales@math.utexas.edu}

\newtheorem{theorem}{Theorem}[section]
\newtheorem{lemma}{Lemma}[section]

\newtheorem{remark}{Remark}[section]

\newcommand{\bbr}{\mathbb R}
\newcommand{\bbs}{\mathbb S}

\newcommand{\bbp} {\mathbb P}

\numberwithin{figure}{section}
\newcommand{\beq}{\begin{equation}}
\newcommand{\eeq}{\end{equation}}
\newcommand{\bsp}{\begin{split}}
\newcommand{\esp}{\end{split}}

\def\eps{\varepsilon }

\newcommand\adots{\mathinner{\mkern2mu\raise1pt\hbox{.}
\mkern3mu\raise4pt\hbox{.}\mkern1mu\raise7pt\hbox{.}}}

\def\charf {\mbox{{\text 1}\kern-.30em {\text l}}}

\begin{document}

\date{\today}


\thanks{\textbf{Acknowledgment.} The authors thank Prof. Alessio Figalli for valuable comments.
}

\begin{abstract}
We consider a spatially homogeneous Kolmogorov-Vicsek model
in two dimensions, which describes the alignment dynamics of
self-driven stochastic particles that move on the plane at a constant
speed, under space-homogeneity. 
In \cite{F-K-M}, Alessio Figalli and the authors have shown the existence of global
weak solutions for this two-dimensional model.
However, no time-asymptotic behavior has been obtained for the two-dimensional case, due to the failure of the celebrated Bakery and Emery condition for the logarithmic Sobolev inequality. We prove exponential convergence (with quantitative rate) of the weak solutions  towards a Fisher-von Mises distribution, using a new condition for the logarithmic Sobolev inequality.
\end{abstract}
\maketitle \centerline{\date}

\section{Introduction and Main results}
Recently, the stochastic Vicsek model has
received extensive attention in 
mathematical topics such as mean-field limits, 
hydrodynamic limits, 
and phase transitions \cite{A-H,B-C-C,C-K-J-R-F, D-F-L-2, D-F-L-1,D-M,F-L,G-C}. In this article, 
we study a time-asymptotic behavior of the so-called 
Kolmogorov-Vicsek model in two space dimensions. Such a model is governed by a nonlinear, nonlocal Fokker-Planck equation, which describes self-propelled stochastic particles moving on a plane with unit speed:
\begin{align}
\begin{aligned}\label{main} 
&\partial_{t}\rho= \Delta_{\omega}\rho-\nabla_{\omega}\cdot\Big(\rho\,\bbp_{\omega^{\perp}}\Omega_{\rho}\Big),\\
&\Omega_{\rho}=\frac{J_{\rho}}{|J_{\rho}|}, \quad J_{\rho}=\int_{\bbs^{1}}\omega\,\rho\, d\omega,
\end{aligned}
\end{align}
where $\rho(t,\omega)$ is a probability 
density function at time $t$ with direction $\omega\in\bbs^{1}$ 
(unit circle of $\bbr^{2}$), and the operators 
$\nabla_{\omega}$ and $\Delta_{\omega}$ denote the gradient and 
the Laplace-Beltrami operator on the circle $\bbs^{1}$.
The force field $\bbp_{\omega^{\perp}}\Omega_{\rho}$ denotes 
the projection of the unit vector $\Omega_{\rho}$ onto the normal 
plane to $\omega$, i.e., $\bbp_{\omega^{\perp}}\Omega_{\rho}:=(Id-\omega\otimes\omega)\Omega_{\rho}$, which describes the mean-field force that governs 
the orientational interaction of self-driven particles by aligning them 
with the direction $\Omega_{\rho}$ determined by the flux $J_{\rho}$.

This model \eqref{main} is a spatially homogeneous version of the kinetic Kolmogorov-Vicsek model \cite{D-M,Ga-Ka}, which was formally derived by Degond and Motsch \cite{D-M} as a mean-field limit of the discrete Vicsek model \cite{A-H, C-K-J-R-F, G-C, Vicsek} with stochastic dynamics. Bolley, Ca$\tilde{\mbox{n}}$izo and Carrillo \cite{B-C-C} have rigorously justified the mean-field limit 
when the unit vector $\Omega_{\rho}$ in the force term of 
\eqref{main}  is replaced by a more regular vector-field. As a study on phase transition, Degond, Frouvelle and Liu \cite{D-F-L-1} provided a complete and rigorous description of phase transitions when $\Omega_{\rho}$
is replaced by $\nu(|J_{\rho}|)\Omega_{\rho}$, and there is a 
noise intensity $\tau(|J_{\rho}|)$ 
in front of $\Delta_{\omega}\rho$, 
where the functions $\nu$ and $\tau$  are Lipschitz, bounded,  and satisfy that $|J_{\rho}|\mapsto \nu(|J_{\rho}|)/|J_{\rho}|$ and $|J_{\rho}|\mapsto \tau(|J_{\rho}|)$. It turns out that their modification leads to the appearance of phase transitions such as the number and nature of equilibria, stability, convergence rate, phase diagram and hysteresis, which depend on the ratio between $\nu$ and $\tau$. We see that the assumptions of $\nu$ remove the singularity of $\Omega_{\rho}$ because $\nu(|J_{\rho}|)\Omega_{\rho}\to 0$ as $|J_{\rho}|\to 0$. 
This phase transition problem has been studied as well in \cite{A-H, C-K-J-R-F, D-F-L-2,D-F-L-1, F-L, G-C}.

As a result on well-posedness of kinetic Vicsek model, Gamba and 
the first author \cite{Ga-Ka}
recently proved the existence and uniqueness
of weak solutions to the spatially inhomogeneous
Vicsek model under the a priori assumption of
positivity of momentum, without handling the stability issues,
whose difficulty is mainly coming from the facts that
the momentum is not conserved and 
no dissipative energy functional has been found for the model. For a study for its numerical scheme, we refer to \cite{G-H-M}. Frouvelle and Liu \cite{F-L} have shown the well-posedness in the spatially homogeneous case with a more regular vector-field, i.e., $\bbp_{\omega^{\perp}}J_{\rho}$ instead of $\bbp_{\omega^{\perp}}\Omega_{\rho}$. 
Concerning studies on hydrodynamic descriptions of kinetic Vicsek model, we refer to \cite{D-F-L-2, D-F-L-1, D-M, D-M-2, D-Y, F};
see also \cite{Bo-Ca, D-D-M, D-F-M, H-J-K} for other related studies.

In \cite{F-K-M}, the authors have shown the global-in-time existence of weak solutions to the 
two-dimensional model \eqref{main}, and short-time stability in 2-Wasserstein distance, whereas they have proved the convergence of the weak solutions towards Fisher-von Mises distribution in the higher-dimensional case for \eqref{main}, that is, the space dimension is bigger than two. In order to show the exponential convergence to steady state, they have used the logarithmic Sobolev inequality based on the celebrated criterion of Bakry and Emery \cite{B-E} (see Section 2), which requires the constraint on the space dimension. 

In this article, we prove that the weak solution of \eqref{main} exponentially converges towards the Fisher-von Mises distribution as a stationary state.\\
Notice that since 
$\bbp_{\omega^{\perp}}\Omega_{\rho}=\nabla_{\omega}(\omega\cdot\Omega_{\rho})$, the equation \eqref{main} can be rewritten in the form:
\beq\label{main-0}
\partial_{t}\rho= \nabla_{\omega}\cdot\Big(\rho\,\nabla_{\omega}\,(\log\rho-\omega\cdot\Omega_{\rho})\Big),
\eeq
which can be regarded as a gradient flow with respect to the Wasserstein distance of the free energy functional 
\beq\label{free-e}
\mathcal{E}(\rho)=\int_{\bbs^{1}}\rho\log\rho \hspace{1mm} d\omega-|J_{\rho}|.
\eeq
We can easily see that the equilibrium states of \eqref{main-0} have the form of the Fisher-von Mises distributions: for any given $\Omega\in\bbs^{1}$,
these are given by
\[
M_{\Omega}(\omega):=C_M e^{\omega\cdot\Omega},
\]
where $C_M$ is the following positive constant
\begin{equation}
\label{eq:CM}
C_M=\frac{1}{\int_{\bbs^{1}}e^{\omega\cdot\Omega}\,d\omega},
\end{equation}
so that $M_{\Omega}$ is a probability density function in $\mathcal{P}(\mathbb{S}^{1});$ the space of probability measures in $\mathbb{S}^{1}$.

As in \cite{F-K-M}, we will show the time-asymptotic behavior using the relative entropy with respect to the Fisher-von Mises distribution $M_{\Omega_{\rho}}$ defined by
\[
H(\rho | M_{\Omega_{\rho}}):=\int_{\mathbb{S}^{1}}\log \Big(\frac{\rho}{M_{\Omega_{\rho}}}\Big) \rho \hspace{1mm} d\omega,
\]
which actually control the $L^1$-distance between $\rho$ and $M_{\Omega_{\rho}}$. 
We  show that $H(\rho | M_{\Omega_{\rho}})$ decays exponentially. The proof of such a decay relies on two main 
estimates. The first one is a localized version of the logarithmic Sobolev inequality which we prove in section 2.
The second one is a growth control of the dissipation 
given in Lemma \ref{lem-DC}. The main heuristic idea for such control  is coming from the final dimensional identity
\[
\frac{d}{dt}\langle\nabla_{\gamma(t)}E,\nabla_{\gamma(t)}E\rangle=-2\langle\nabla_{\gamma(t)}^{2}E\nabla_{\gamma(t)}E,\nabla_{\gamma(t)}E\rangle,
\hspace{1em}\forall t\geq 0,
\]
which holds for any $E$ in $\mathcal{C}^{2}(\mathbb{R}^n)$ and any $\gamma$ in $\mathcal{C}^{1}([0,\infty),\mathbb{R}^{n})$
satisfying $\dot{\gamma}=-\nabla_{\gamma}E.$\\
Indeed, in the setting of the formal Calculus introduced by Otto and Villani (see \cite[Section 3]{Otto-Villani}),
such a identity corresponds to the connection between \eqref{claim-0} and the formal expression for the Wasserstein Hessian of the free energy \eqref{free-e} appeared in \cite[Lemma 3.1]{F-K-M}. 

We now state the main results of the paper.
\begin{theorem} \label{thm-converge}
Let $\rho_{0}\in\mathcal{P}(\bbs^{1})$ be an initial probability measure satisfying
\beq\label{ini-assume}
\rho_0>0,\quad |J_{\rho_{0}}|>0, \quad \int_{\bbs^{1}}\rho_0\log\rho_0 \,d\omega<\infty.
\eeq
Then, the equation \eqref{main} has a unique weak solution $\rho\in L_{loc}^{2}([0,\infty),W^{1,1}(\bbs^{1}))$, which is weakly continuous in time, and satisfies time-asymptotic behaviors as follows:\\
1) There exist a constant unit vector $\Omega_{\infty} \in \bbs^{1}$ such that for all $t>0$,
\beq\label{thm-ineq}
 \|\rho_t -  M_{\Omega_{\infty}}\|_{L^1(\bbs^{1})} \lesssim \Big(\int_{\bbs^{1}}\rho_0\log\rho_0\, d\omega+1-\log C_{M}\Big) e^{-Bt}.
\eeq
Here, the exponential decay rate $B$ depends on initial data as
\[
B=\bigg(\frac{|J_{\rho_0}|e^{-T_{0}}}{2|J_{\rho_0}|e^{-T_{0}}+2} \bigg( {\rm \exp}\bigg[\bigg(2+\frac{2}{|J_{\rho_{0}}|e^{-T_{0}}}\bigg)T_{0}\bigg] -1\bigg) + C_*{\rm \exp}\bigg[\bigg(2+\frac{2}{|J_{\rho_{0}}|e^{-T_{0}}}\bigg)T_{0}\bigg]   \bigg)^{-1},
\]
where $T_0>0$ is some constant depending on initial data with $T_0\lesssim H(\rho_0 | M_{\Omega_{\rho_0}})$, and
\[
C_{*}:= \frac{2\pi^{2}e^{2(1+|\log C_{M}|)}(1+\frac{1}{15}\varepsilon_{*})}{1-\frac{7}{6}\varepsilon_{*}}, 
\]
for any fixed positive constant $\varepsilon_{*}\leq \frac{1}{10}.$\\
2) For any $t\geq T_{0}$, we have
\[
\|\rho_t -  M_{\Omega_{\infty}}\|_{L^1(\bbs^{1})} \lesssim \Big(\int_{\bbs^{1}}\rho_0\log\rho_0\, d\omega+1-\log C_{M}\Big) e^{-C_*^{-1}(t-T_{0})}.
\]
3) For any $\varepsilon>0$, there exists $\delta>0$ such that if  $H(\rho_{0}|M_{\Omega_{\rho_{0}}})<\delta$, then for all $t>0$,
\[
\|\rho_t -  M_{\Omega_{\infty}}\|_{L^1(\bbs^{1})} \lesssim \Big(\int_{\bbs^{1}}\rho_0\log\rho_0\, d\omega+1-\log C_{M}\Big) 
\exp{\Big[-\Big(\frac{1}{2\pi^{2}e^{2(1+|\log C_{M}|)}}-\eps\Big)t\Big]}.
\]
\end{theorem}
\begin{remark}
The estimates \eqref{thm-ineq} in Theorem \eqref{thm-converge} represent the exponential convergence of weak solutions towards some steady state $M_{\Omega_{\infty}}$, where it is not clear how to determine the vector $\Omega_{\infty}$ from the initial data $\rho_0$, because the momentum $J_{\rho_t}$ is not conserved in time.
\end{remark}

\begin{remark}
Following the proof in Section 3, the estimates \eqref{thm-ineq} under the above three conditions 1)-3) are straightforward results of the more detailed estimates (see \eqref{main-ineq}):
\begin{align}
\begin{aligned}\label{remark-ineq}
\|\rho_{t}-M_{\Omega_{\infty}}\|_{L^{1}(\mathbb{S}^{1})}
&\le \Big(\int_{\bbs^{1}}\rho_0\log\rho_0\,d\omega +1-\log C_{M}\Big)\\
&\quad\times\bigg(e^{-\int_{0}^{t} B(s)ds}+C\int_{t}^{\infty}e^{-\int_{0}^{r}B(s)ds}dr \bigg), \quad \forall t>0,
\end{aligned}
\end{align}
where $C>0$ is some constant and $B(t)$ is a positive function defined by (with notation $f_+:=\max(f,0)$)
\begin{multline*}
B(t):=\bigg(\frac{|J_{\rho_0}|e^{-T_{0}}}{2|J_{\rho_0}|e^{-T_{0}}+2}\bigg({\rm \exp}\bigg[\bigg(2+\frac{2}{|J_{\rho_{0}}|e^{-T_{0}}}\bigg)T_{0}\bigg]-{\rm \exp}\bigg[\bigg(2+\frac{2}{|J_{\rho_{0}}|e^{-T_{0}}}\bigg)t\bigg]\bigg)_{+}\\
+C_* \exp\bigg[\bigg(2+\frac{2}{|J_{\rho_{0}}|e^{-T_{0}}}\bigg)(T_{0}-t)_{+}\bigg]\bigg)^{-1}.
\end{multline*}
Here, $T_0>0$ is some constant such that
\[
 T_{0}\leq 2H(\rho_0 | M_{\Omega_{\rho_0}})\Big[\min\Big(C_M^2e^{-2}\eps_*^2, \frac{L}{2C_* }C_M^2e^{-2}\eps_*^2, m  C_*^{-1} \Big)\Big]^{-1},
\]
where $C_*, \eps_*$ are the constants appeared in Theorem \ref{thm-converge}, and 
\[
L:=\bigg(2+\frac{4}{m}\bigg)^{-1}\log 2,
\]
where $m$ denotes a strength of momentum of the Fisher-von Mises distribution, i.e.,
\[
m:=\Big|\int_{\bbs^{1}}\omega M_{\Omega_{\rho}} \,d\omega\Big|,
\]
which is a constant (see \cite[Appendix]{F-K-M}).
\end{remark}
\begin{remark}
In the proof of Theorem \ref{thm-converge}, we first obtain the exponential convergence of $H(\rho_{t}|M_{\Omega_{t}})$ as
\beq\label{remark-H}
H(\rho_{t}|M_{\Omega_{\rho_t}})\leq H(\rho_{0}|M_{\Omega_{\rho_{0}}})e^{-\int_{0}^{t} B(s)ds}.
\eeq
However, the Fisher-von Mises distribution $M_{\Omega_{\rho_t}}$ is not constant in time, because of no conservation of momentum $J_{\rho_t}$. In fact, we show that $\Omega_{\rho_t}$ stabilizes to a fixed vector as $t\to \infty$, observing that
$|\frac{d}{dt}J_{\rho_t}|$ vanishes as $t\to \infty$ thanks to the decay estimate \eqref{remark-H}.
\end{remark}
\begin{remark}
Notice that since 
\[
\int_{\bbs^1}\rho_0\log\rho_0\,d\omega -1-\log C_{M} \le H(\rho_{0}\mid M_{\Omega_{\rho_0}})\le \int_{\bbs^1}\rho_0\log\rho_0\,d\omega +1-\log C_{M},
\]
we see that $H(\rho_{0}\mid M_{\Omega_{\rho_0}})<\infty$ if and only if $\int_{\bbs^{1}}\rho_0\log\rho_0 \,d\omega<\infty$. Therefore, the initial condition $\int_{\bbs^{1}}\rho_0\log\rho_0 \,d\omega<\infty$ in \eqref{ini-assume} can be replaced by $H(\rho_{0}\mid M_{\Omega_{\rho_0}})<\infty$.
\end{remark}

In \cite{F-K-M}, the authors have proven that for any dimension $n\ge1$, the model \eqref{main} on $n$-dimensional sphere $\bbs^n$ has a  weak solution $\rho \in L_{loc}^{2}([0,\infty),W^{1,1}(\bbs^{n}))$ that satisfies 
\beq\label{J-positive}
|J_{\rho_t}|\geq |J_{\rho_{0}}| e^{-nt}, \quad \forall t>0.
\eeq
Additionally, we showed short-time stability in 2-Wasserstein distance $W_2$ as follows: for any probability measure $\rho_{0}$ and $\bar{\rho}_{0}$ satisfying \eqref{ini-assume}
  and
  \[
   W_{2}(\rho_{0},\bar\rho_{0})\leq \frac{|J_{\rho_{0}}|}{16},
  \]
 there exists $\delta>0$ such that any two solutions $\rho_t$ and $\bar\rho_t$ of \eqref{main} starting from $\rho_{0}$ and $\bar{\rho}_{0}$ satisfies
  \begin{equation}\label{stab}
  W_{2}(\rho_{t},\bar{\rho}_{t}) \leq \exp{\bigg[\Big(2-n+\frac{2}{|J_{\rho_{0}}|}\Big)t\bigg]}W_{2}(\rho_{0},\bar{\rho}_{0}),\quad \forall t<\delta,
  \end{equation}
where the short time $\delta$ is explicitly found as 
\beq\label{w-delta}
\delta=\frac{|J_{\rho_{0}}|^{4}}{2^{8}\mbox{max}\{H(\rho_{0}|M_{\Omega_{\rho_{0}}}),H(\bar\rho_{0}|M_{\Omega_{\bar\rho_{0}}})\}}.
\eeq
In fact, for this stability estimate, the authors have not used the logarithmic Sobolev inequality based on the criterion of Bakry and Emery. Therefore, the above stability still holds when $n=1$, whose case will be added in \cite{F-K-M}. On the other hand, the uniqueness of the weak solutions has been proven in the case of $n\ge 2$, by using the stability estimate \eqref{stab} up to $\delta$ in \eqref{w-delta} together with the exponential convergence towards Fisher-von Mises distribution. 

Therefore, once we get the exponential convergence \eqref{main-ineq} in Theorem \ref{thm-converge}, the uniqueness of the weak solution to \eqref{main} holds true by the same argument in \cite{F-K-M} as mentioned above. Recently, we recognized that we can prove the uniqueness, based on \eqref{stab} and the energy method with the parabolic regularity without using the large-time behavior. We give its proof in Appendix B as an another proof of the uniqueness.  \\

The paper is organized as follows. In the next section, we present a localized version for logarithmic Sobolev inequality, which is crucially used in the proof of Theorem \ref{thm-converge} in Section 3.

\section{Logarithmic Sobolev inequality}
\setcounter{equation}{0}
In this section, we present a simple condition for
the logarithmic Sobolev inequality on $\mathbb{S}^{1}$ equipped with the ambient metric. 
Such condition is one of the crucial ingredients that we use to show the exponential decay \eqref{main-ineq}. 
We first fix a 
reference probability measure $e^{-\Psi(\omega)} d\omega$ such that 
$\Psi\in C^1(\mathbb{S}^{1})$. Then, for a given probability 
measure $\rho \, d\omega$, we define its relative entropy with respect to $e^{-\Psi(\omega)} d\omega$ by
\begin{align*}
\begin{aligned} 
H(\rho | e^{-\Psi} )=\int_{\mathbb{S}^{1}}\log \Big(\frac{\rho}{e^{-\Psi}}\Big) \rho \hspace{1mm} d\omega,
\end{aligned}
\end{align*}
and the relative Fisher information by 
\[
I(\rho | e^{-\Psi})=\int_{\mathbb{S}^{1}}\Big|\nabla \log \frac{\rho}{e^{-\Psi}} \Big|^2 \rho \, d\omega,
\]
where $\nabla$ denotes the gradient on $\mathbb{S}^{1}$.\\
We recall the celebrated criterion of Bakry and Emery \cite{B-E} for a logarithmic Sobolev inequality as follows: If there exists a constant $\alpha>0$ such that
\beq\label{BE-ass}
D^2 \Psi + \mbox{Ricci curvature tensor on $\mathbb{S}^{1}$} \ge \alpha I_n, 
\eeq
then the probability measure $e^{-\Psi} d\omega$ satisfies a logarithmic Sobolev inequality with $\alpha$, i.e., for all probability measure $\rho \hspace{1mm} d\omega$ we have
\[
H(\rho | e^{-\Psi} )\le \frac{1}{2\alpha} I(\rho | e^{-\Psi}).
\]
In \cite{F-K-M}, we proved the exponential convergence estimate like \eqref{main-ineq} with explicit decay rate, using the above criterion \eqref{BE-ass} together with the facts that the Ricci curvature tensor of $\bbs^{n}$ is $(n-1)I_{n}$, and the logarithmic Sobolev inequality is stable under bounded perturbations (see \cite{H-S, Otto-Villani}). 
However, if we consider the 1-dimensional sphere $\bbs^{1}$,
its Ricci curvature vanishes. Therefore, the condition \eqref{BE-ass} is not satisfied anymore when $\Psi$ takes the form $\Psi(\omega)=-\omega\cdot\Omega,$ as  a Fisher-von Mises distribution. 
We will overcome this difficulty by imposing a smallness condition of the $L^{\infty}$
distance between $\rho$ and $e^{-\Psi}$ in the following Lemma.
\begin{lemma}\label{lem-LSI}
Let $\rho$ and $e^{-\Psi}$ be probability measures on  $\mathbb{S}^{1}$ such that $\rho, \Psi\in C^1(\mathbb{S}^{1})$, and $|\Psi| \le\lambda$, for some constant $\lambda$. For any fixed positive constant $\eps_*\le\frac{1}{10}$, if
\[
\|\rho e^{\Psi}-1\|_{L^{\infty}(\mathbb{S}^{1})}\le\eps_*,
\]
then,
\beq\label{LSI-1}
H(\rho | e^{-\Psi} )\le \frac{2\pi^{2}e^{2\lambda}(1+\frac{1}{15}\varepsilon_{*})}{1-\frac{7}{6}\varepsilon_{*}} I(\rho | e^{-\Psi}).
\eeq
\end{lemma}
\begin{proof}
We set $f:=\rho e^{\Psi}-1$, $\gamma (s) := (1+sf)e^{-\Psi}$ for $s\in[0,1]$.  Then we define 
\[
H(\gamma(s) | e^{-\Psi} )=\int_{\mathbb{S}^{1}} (1+sf)e^{-\Psi}\log (1+sf) d\omega,
\]
and 
\[
I(\gamma(s) | e^{-\Psi} )=\int_{\mathbb{S}^{1}} (1+sf)e^{-\Psi}|\nabla \log (1+sf)|^2 d\omega.
\]
for every $s$  in$[0,1].$ \\
We will use the Taylor theorem and the Poincar\'e inequality to show \eqref{LSI-1}. First of all, we see that
\[
H(\gamma(0) | e^{-\Psi} )=0=I(\gamma(0) | e^{-\Psi}).
\]
A straightforward computation yields
\begin{align*}
\begin{aligned}
&\frac{d}{ds}H(\gamma(s) | e^{-\Psi} ) = \int_{\mathbb{S}^{1}} (1+\log (1+sf)) f e^{-\Psi} d\omega, \\
&\frac{d^2}{ds^2}H(\gamma(s) | e^{-\Psi} ) = \int_{\mathbb{S}^{1}} \frac{f^2 }{1+sf}e^{-\Psi} d\omega, \\
&\frac{d^3}{ds^3}H(\gamma(s) | e^{-\Psi} ) = -\int_{\mathbb{S}^{1}} \frac{f^3 }{(1+sf)^2}e^{-\Psi} d\omega.
\end{aligned}
\end{align*}
Setting $g(s)=1+sf$, by direct computation, we obtain
\begin{align*}
\begin{aligned}
&\frac{d}{ds}I(\gamma(s) | e^{-\Psi} ) = \int_{\mathbb{S}^{1}} \nabla \log (g(s)) \Big[2\nabla\Big(\frac{g'(s)}{g(s)}\Big)g(s)+\nabla\log(g(s)) g'(s)\Big] e^{-\Psi} d\omega,\\
&\frac{d^2}{ds^2}I(\gamma(s) | e^{-\Psi} ) = \int_{\mathbb{S}^{1}} \nabla \Big(\frac{g'(s)}{g(s)}\Big)  \Big[2\nabla\Big(\frac{g'(s)}{g(s)}\Big)g(s)+\nabla\log(g(s))g'(s)\Big] e^{-\Psi} d\omega\\
&\qquad +  \int_{\mathbb{S}^{1}} \nabla \log (g(s)) \Big[-2\nabla\Big(\frac{|g'(s)|^2}{g(s)^2}\Big)g(s)+3\nabla\Big(\frac{g'(s)}{g(s)}\Big)g'(s) \Big] e^{-\Psi} d\omega,\\
&\frac{d^3}{ds^3}I(\gamma(s) | e^{-\Psi} ) = -\int_{\mathbb{S}^{1}} \nabla\Big(\frac{|g'(s)|^2}{g(s)^2}\Big) \Big[2\nabla\Big(\frac{g'(s)}{g(s)}\Big)g(s)+\nabla\log(g(s))g'(s)\Big] e^{-\Psi} d\omega\\
&\qquad +  2\int_{\mathbb{S}^{1}} \nabla \Big(\frac{g'(s)}{g(s)}\Big)  \Big[-2\nabla\Big(\frac{|g'(s)|^2}{g(s)^2}\Big)g(s)+3\nabla\Big(\frac{g'(s)}{g(s)}\Big)g'(s) \Big] e^{-\Psi} d\omega\\
&\qquad +  \int_{\mathbb{S}^{1}} \nabla\log (g(s)) \Big[4\nabla\Big(\frac{|g'(s)|^3}{g(s)^3}\Big)g(s)-5\nabla\Big(\frac{|g'(s)|^2}{g(s)^2}\Big)g'(s) \Big] e^{-\Psi} d\omega,
\end{aligned}
\end{align*}
where we have the fact that used $g''(s)=0$.\\
Since
\begin{align*}
\begin{aligned}
&\frac{d}{ds}\Big|_{s=0}H(\gamma(s) | e^{-\Psi} ) = \int_{\mathbb{S}^{1}}  f e^{-\Psi} d\omega = \int_{\mathbb{S}^{1}} (\rho - e^{-\Psi}) d\omega=0,\\
&\frac{d}{ds}\Big|_{s=0} I(\gamma(s) | e^{-\Psi} ) = 0,
\end{aligned}
\end{align*}
it follows from Taylor theorem that there exists $s_1, s_2\in (0,1)$ such that
\begin{align}
\begin{aligned}\label{taylor}
&H(\rho | e^{-\Psi} ) =H(\gamma(1) | e^{-\Psi} ) = \frac{1}{2}\int_{\mathbb{S}^{1}}  f^2 e^{-\Psi} d\omega +\frac{1}{6}\frac{d^3}{ds^3}\Big|_{s=s_1}H(\gamma(s) | e^{-\Psi} ),\\
&I(\rho | e^{-\Psi} ) =I(\gamma(1) | e^{-\Psi} ) = \int_{\mathbb{S}^{1}}  |\nabla f|^2 e^{-\Psi} d\omega +\frac{1}{6}\frac{d^3}{ds^3}\Big|_{s=s_2}I(\gamma(s) | e^{-\Psi} ).
\end{aligned}
\end{align}
Since $\int_{\bbs^1}  f e^{-\Psi} d\omega =0$ and $e^{-\Psi}>0$, there exists $\omega_0\in\bbs^1$ such that $f(\omega_0)=0$. Then, using $|\Psi| \le\lambda$, we have a Poincar\'e inequality:
\begin{align}\label{poincare}
\begin{aligned}
\int_{\mathbb{S}^{1}}  f^2 e^{-\Psi} d\omega  &\le e^{\lambda}\int_{\mathbb{S}^{1}}  f^2 d\omega = e^{\lambda}\int_{\mathbb{S}^{1}} (f(\omega)-f(\omega_0))^2 d\omega = e^{\lambda}\int_{\mathbb{S}^{1}} \Big(\int_{\omega_0}^{\omega}\nabla f \Big)^2 d\omega \\
&\le 4e^{\lambda}\pi^2\int_{\mathbb{S}^{1}} |\nabla f|^2 d\omega \le 4e^{2\lambda}\pi^2\int_{\mathbb{S}^{1}} |\nabla f|^2 e^{-\Psi} d\omega.
\end{aligned}
\end{align}
A straightforward computation with $\|f\|_{L^{\infty}(\bbs^1)}\le\eps_*\le \frac{1}{10}$ gives the estimates for the third order terms in \eqref{taylor} as follows:
\[
 \Big|\frac{d^3}{ds^3}\Big|_{s=s_1}H(\gamma(s) | e^{-\Psi})\Big| \le \frac{1}{5}\eps_* \int_{\mathbb{S}^{1}}  f^2 e^{-\Psi} d\omega
\]
and
\[
 \Big|\frac{d^3}{ds^3}\Big|_{s=s_2}I(\gamma(s) | e^{-\Psi}) \Big| \le  42\eps_* \int_{\mathbb{S}^{1}}  |\nabla f|^2 e^{-\Psi} d\omega, 
\]
Therefore, using \eqref{taylor}, \eqref{poincare}, and the above estimates we get
\begin{multline*}
H(\rho|e^{-\Psi})
\leq\frac{1}{2}\Big(1+\frac{1}{15}\eps_{*}\Big)\int_{\mathbb{S}^{1}}f^{2}e^{-\Psi}d\omega\\ 
\le2\pi^{2}e^{2\lambda}\Big(1+\frac{1}{15}\eps_{*}\Big)\int_{\mathbb{S}^{1}}|\nabla f|^{2}e^{-\Psi}d\omega\le\frac{2\pi^{2}e^{2\lambda}\Big(1+\frac{1}{15}\eps_{*}\Big)}{1-\frac{7}{6}\eps_{*}}I(\rho|e^{-\Psi}).
\end{multline*}
\end{proof}

\begin{remark}
 In the proof of Lemma \ref{lem-LSI}, we see that the Poincar\'e 
inequality implies the logarithmic Sobolev inequality \eqref{LSI-1},
provided $L^{\infty}$ distance between $\rho$ and $e^{-\Psi}$ is
suitably small on $\mathbb{S}^{1}$. In general, we refer to \cite{Rothaus} (see also \cite{Otto-Villani}) on the implication of the Poincar\'e inequality as a linearization of the logarithmic Sobolev inequality, on a compact Riemannian manifold.
\end{remark}

\section{Proof of Theorem \ref{thm-converge}}
\setcounter{equation}{0}
We begin by observing that a straightforward computation with \eqref{main-0} implies 
\begin{align}
\begin{aligned}\label{ineq-1}
\frac{d}{dt}H(\rho_{t} 
| M_{\Omega_{\rho_{t}}}) = 
- I(\rho_{t} | M_{\Omega_{\rho_{t}}}).
\end{aligned}
\end{align}
We first show that $H(\rho | M_{\Omega_{\rho}})$ exponentially decays after a suitably large time using \eqref{ineq-1} and Lemma \ref{lem-LSI}. Notice that the weak solution $\rho$ to \eqref{main} becomes smooth instantaneously thanks to the parabolic regularization, therefore we can consider smoothness of weak solutions $\rho$ to \eqref{main} for the decay estimates. Moreover, since $\rho_0>0$ by the assumption, we see that the weak solution $\rho$ of the parabolic equation \eqref{main} is positive for all time, i.e., $\rho_t>0$ for all $t>0$.  \\
Since $\bbs^1$ is one-dimensional manifold, we easily see that $\|\rho M_{\Omega_{\rho}}^{-1} -1\|_{L^{\infty}(\bbs^1)}$ is controlled by the dissipation $I(\rho | M_{\Omega_{\rho}})$ as follows:
\beq\label{Sobo}
\|\rho M_{\Omega_{\rho}}^{-1} -1\|_{L^{\infty}(\bbs^1)}\le C_M^{-1}e\sqrt{I(\rho | M_{\Omega_{\rho}})}.
\eeq  
Indeed, since $\int_{\bbs^1}\rho d\omega=\int_{\bbs^1}M_{\Omega_{\rho}} d\omega=1$, there exists $\omega_0\in\bbs^1$ such that $\rho(\omega_0)=M_{\Omega_{\rho}}(\omega_0)$, which together with $C_M e^{-1}\le  M_{\Omega_{\rho}} \le C_M e$ imply that
\begin{align*}
\begin{aligned}
\rho M_{\Omega_{\rho}}^{-1}(\omega) -1 &=\rho M_{\Omega_{\rho}}^{-1}(\omega) -\rho M_{\Omega_{\rho}}^{-1}(\omega_0) =\int_{\omega_0}^{\omega} \nabla (\rho M_{\Omega_{\rho}}^{-1})\hspace{1mm}d\omega \\
&\le \sqrt{ C_M^{-1}e}\Big(\int_{\bbs^1} \frac{|\nabla (\rho M_{\Omega_{\rho}}^{-1})|^2}{\rho M_{\Omega_{\rho}}^{-1}} d\omega\Big)^{1/2} 
\le  C_M^{-1}e\Big(\int_{\bbs^1} \Big|\frac{\nabla (\rho M_{\Omega_{\rho}}^{-1})}{\rho M_{\Omega_{\rho}}^{-1}}\Big|^2 \rho\hspace{1mm}d\omega\Big)^{1/2}\\
&= C_M^{-1}e\sqrt{I(\rho | M_{\Omega_{\rho}})}.
\end{aligned}
\end{align*}
We will show that for $\eps_*$ appeared in Lemma \ref{lem-LSI}, there exists $T_0>0$ such that
\beq\label{diss-small}
I(\rho_t | M_{\Omega_{\rho_t}}) \le C_M^2 e^{-2} \eps_*^2,\quad  \forall t\ge T_0,
\eeq
therefore we use Lemma \ref{lem-LSI} together with \eqref{ineq-1} to get the exponential decay of $H(\rho | M_{\Omega_{\rho}})$.\\

Before showing \eqref{diss-small}, we first prove the following lemma on growth estimates of the dissipation $I(\rho | M_{\Omega_{\rho}})$.\\

\begin{lemma}\label{lem-DC}
Let $\rho_{0}\in\mathcal{P}(\bbs^{1})$ be an initial probability measure satisfying \eqref{ini-assume}. Then, the solution $\rho$ of \eqref{main} starting from $\rho_0$ satisfies that for any $t\ge s> 0$,
\beq\label{d-ineq}
I(\rho_t | M_{\Omega_{\rho_t}})\leq I(\rho_s | M_{\Omega_{\rho_s}}) e^{C(t,s)(t-s)},
\eeq
where
\[
C(t,s)=2+\frac{2}{\min_{r\in[s,t]}|J_{\rho_{r}}|}.
\]
\end{lemma}
\begin{proof}
We claim that
\begin{align}
\begin{aligned}\label{claim-0}
\frac{d}{dt}I(\rho_t | M_{\Omega_{\rho_t}})&=-2\int_{\bbs^1} \mbox{tr}\Big(\Big[D^{2}\log\frac{\rho_{t}}{e^{\omega\cdot\Omega_{\rho_{t}}}}\Big]^{T}D^{2}\log\frac{\rho_{t}}{e^{\omega\cdot\Omega_{\rho_{t}}}}\Big)\rho_{t} \hspace{1mm}d\omega\\
&\quad+2\int_{\bbs^1}\nabla\log\frac{\rho_{t}}{e^{\omega\cdot\Omega_{\rho_{t}}}}D^{2}(\omega\cdot\Omega_{\rho_{t}})\nabla\log\frac{\rho_{t}}{e^{\omega\cdot\Omega_{\rho_{t}}}}\rho_{t} \hspace{1mm}d\omega\\
&\quad+\frac{2}{|J_{\rho_{t}}|}\bigg(\bigg|\int_{\bbs^1}\nabla\log\frac{\rho_{t}}{e^{\omega\cdot\Omega_{\rho_{t}}}}\rho_{t} d\omega\bigg|^{2}-\bigg|\int_{\bbs^1}\Omega_{\rho_{t}}\cdot\nabla\log\frac{\rho_{t}}{e^{\omega\cdot\Omega_{\rho_{t}}}}\rho_{t} d\omega\bigg|^{2}\bigg).
\end{aligned} 
\end{align}
Once we show the above equality, since the first term in r.h.s. of \eqref{claim-0} is non-positive and the Hessian of the map $\omega \mapsto \omega\cdot\Omega_{\rho_t}$ has norm bounded by $1$, we have
\begin{align*}
\begin{aligned}
\frac{d}{dt}I(\rho_t | M_{\Omega_{\rho_t}})&\le \Big(2+\frac{2}{|J_{\rho_{t}}|}\Big)\int_{\bbs^1}\Big|\nabla\log\frac{\rho_{t}}{e^{\omega\cdot\Omega_{\rho_t}}}\Big|^{2}\rho_{t} \hspace{1mm}d\omega\\
&= \Big(2+\frac{2}{|J_{\rho_{t}}|}\Big)I(\rho_t | M_{\Omega_{\rho_t}}),
\end{aligned} 
\end{align*}
which provides the desired inequality \eqref{d-ineq}. Therefore, it remains to prove the claim \eqref{claim-0}.\\
First of all, we separate $I(\rho_t | M_{\Omega_{\rho_t}})$ into two parts as
\begin{align*}
\begin{aligned}
I(\rho_t | M_{\Omega_{\rho_t}})&=\int_{\bbs^1}\nabla\log\frac{\rho_{t}}{e^{\omega\cdot\Omega_{\rho_{t}}}}\cdot\nabla\log\rho_{t}\rho_{t} \hspace{1mm}d\omega
-\int_{\bbs^1}\nabla\log\frac{\rho_{t}}{e^{\omega\cdot\Omega_{\rho_{t}}}}\cdot\nabla(\omega\cdot
\Omega_{\rho_{t}})\rho_{t} \hspace{1mm}d\omega\\ 
&=: I_{1}+I_{2},
\end{aligned} 
\end{align*}
Using the Eq. \eqref{main-0} and integration by parts, we compute $\frac{d}{dt} I_{1}$ as follows:
\begin{align*}
\begin{aligned}
\frac{d}{dt} I_{1} &=\int_{\bbs^1}\nabla\partial_{t}\log\frac{\rho_{t}}{e^{\omega\cdot\Omega_{\rho_{t}}}}\cdot\nabla\rho_{t} \hspace{1mm}d\omega
 +\int_{\bbs^1}\nabla\log\frac{\rho_{t}}{e^{\omega\cdot\Omega_{\rho_{t}}}}\cdot\nabla\partial_{t}\rho_{t} \hspace{1mm}d\omega\\ 
 &=-\int_{\bbs^1}\partial_{t}\log\frac{\rho_{t}}{e^{\omega\cdot\Omega_{\rho_{t}}}}\Delta\rho_{t} \hspace{1mm}d\omega
 -\int_{\bbs^1} \Delta\log\frac{\rho_{t}}{e^{\omega\cdot\Omega_{\rho_{t}}}}\nabla\cdot\Big(\nabla\log\frac{\rho_{t}}{e^{\omega\cdot\Omega_{\rho_{t}}}} \rho_{t}\Big) \hspace{1mm}d\omega\\ 
 &=-\int_{\bbs^1}\frac{e^{\omega\cdot\Omega_{\rho_{t}}}}{\rho_{t}}\partial_{t}\frac{\rho_{t}}{e^{\omega\cdot\Omega_{\rho_{t}}}}\Delta\rho_{t} \hspace{1mm}d\omega
 +\int_{\bbs^1}\nabla\Delta\log\frac{\rho_{t}}{e^{\omega\cdot\Omega_{\rho_{t}}}}\cdot\nabla\log\frac{\rho_{t}}{e^{\omega\cdot\Omega_{\rho_{t}}}} \rho_{t} \hspace{1mm}d\omega\\ 
 &=-\int_{\bbs^1}\frac{\partial_{t}\rho_{t}}{\rho_{t}} \Delta\rho_{t}\hspace{1mm}d\omega+\int\partial_{t}(\omega\cdot\Omega_{\rho_{t}})\Delta\rho_{t}\hspace{1mm}d\omega +\int_{\bbs^1}\nabla\Delta\log\frac{\rho_{t}}{e^{\omega\cdot\Omega_{\rho_{t}}}}\cdot\nabla\log\frac{\rho_{t}}{e^{\omega\cdot\Omega_{\rho_{t}}}} \rho_{t} \hspace{1mm}d\omega.
\end{aligned} 
\end{align*}
Similarly, compute $\frac{d}{dt} I_{2}$ as follow:
\begin{align*}
\begin{aligned}
\frac{d}{dt} I_{2} &=-\int_{\bbs^1}\nabla\partial_{t}\log\frac{\rho_{t}}{e^{\omega\cdot\Omega_{\rho_{t}}}}\cdot\nabla(\omega\cdot
\Omega_{\rho_{t}})\rho_{t} \hspace{1mm}d\omega
-\int_{\bbs^1}\nabla\log\frac{\rho_{t}}{e^{\omega\cdot\Omega_{\rho_{t}}}}\cdot\nabla\partial_{t}(\omega\cdot
\Omega_{\rho_{t}})\rho_{t} \hspace{1mm}d\omega\\
&\quad-\int_{\bbs^1}\nabla\log\frac{\rho_{t}}{e^{\omega\cdot\Omega_{\rho_{t}}}}\cdot\nabla(\omega\cdot
\Omega_{\rho_{t}})\partial_{t}\rho_{t} \hspace{1mm}d\omega\\ 
&=:I_{21}+I_{22}+I_{23},
\end{aligned}
\end{align*}
where the three terms are computed as
\begin{align*}
\begin{aligned}
I_{21}&=\int_{\bbs^1}\partial_{t}\log\frac{\rho_{t}}{e^{\omega\cdot\Omega_{\rho_{t}}}}\nabla\cdot\Big(\nabla(\omega\cdot
\Omega_{\rho_{t}})\rho_{t}\Big) \hspace{1mm}d\omega\\
&=\int_{\bbs^1} \Big(\frac{\partial_{t}\rho_{t}}{\rho_{t}}-\partial_{t}(\omega\cdot\Omega_{\rho_{t}})\Big)\Big(\Delta(\omega\cdot
\Omega_{\rho_{t}})\rho_{t}+\nabla(\omega\cdot
\Omega_{\rho_{t}})\cdot\nabla\rho_{t}\Big) \hspace{1mm}d\omega,\\
I_{22}&=-\int_{\bbs^1}\Big(\nabla\rho_{t}  -\nabla(\omega\cdot\Omega_{\rho_{t}}) \rho_{t}\Big)\cdot\nabla\partial_{t}(\omega\cdot
\Omega_{\rho_{t}}) \hspace{1mm}d\omega\\
&=\int_{\bbs^1}\Big(\Delta\rho_{t}-\Delta(\omega\cdot\Omega_{\rho_{t}})\rho_{t} -\nabla(\omega\cdot\Omega_{\rho_{t}})\cdot\nabla\rho_{t}\Big)\partial_{t}(\omega\cdot\Omega_{\rho_{t}}) \hspace{1mm}d\omega,\\
I_{23}&=-\int_{\bbs^1}\nabla\log\frac{\rho_{t}}{e^{\omega\cdot\Omega_{\rho_{t}}}}\cdot\nabla(\omega\cdot
\Omega_{\rho_{t}})\nabla\cdot\Big(\nabla\log\frac{\rho_{t}}{e^{\omega\cdot\Omega_{\rho_{t}}}} \rho_{t}\Big) \hspace{1mm}d\omega \\
&=\int_{\bbs^1}\nabla\Big(\nabla\log\frac{\rho_{t}}{e^{\omega\cdot\Omega_{\rho_{t}}}}\cdot\nabla(\omega\cdot
\Omega_{\rho_{t}})\Big)\cdot\nabla\log\frac{\rho_{t}}{e^{\omega\cdot\Omega_{\rho_{t}}}} \rho_{t} \hspace{1mm}d\omega 
\end{aligned}
\end{align*}
Combining the above computations, we have 
\begin{align*}
\begin{aligned}
&\frac{d}{dt}I(\rho_t | M_{\Omega_{\rho_t}})\\
&\quad=\int_{\bbs^1}\bigg(\nabla\Delta\log\frac{\rho_{t}}{e^{\omega\cdot\Omega_{\rho_{t}}}}\cdot\nabla\log\frac{\rho_{t}}{e^{\omega\cdot\Omega_{\rho_{t}}}}+\nabla\Big(\nabla\log\frac{\rho_{t}}{e^{\omega\cdot\Omega_{\rho_{t}}}}\cdot\nabla(\omega\cdot\Omega_{\rho_{t}})\Big)\cdot\nabla\log\frac{\rho_{t}}{e^{\omega\cdot\Omega_{\rho_{t}}}}\bigg) \rho_{t} \hspace{1mm}d\omega\\
&\qquad-\int_{\bbs^1}\bigg(\frac{\Delta\rho_{t}}{\rho_{t}}-\Delta(\omega\cdot\Omega_{\rho_{t}})-\frac{\nabla(\omega\cdot\Omega_{\rho_{t}})\cdot\nabla\rho_{t}}{\rho_{t}}\bigg)\partial_{t}\rho_{t} \hspace{1mm}d\omega\\
&\qquad+2\int_{\bbs^1}\Big(\Delta\rho_{t}-\Delta(\omega\cdot\Omega_{\rho_{t}})\rho_{t} -\nabla(\omega\cdot\Omega_{\rho_{t}})\cdot\nabla\rho_{t}\Big)\partial_{t}(\omega\cdot\Omega_{\rho_{t}}) \hspace{1mm}d\omega\\
&\quad=:J_1+J_2+J_3.
\end{aligned}
\end{align*}

We first use \eqref{laplog} and \eqref{main-0} to get
\begin{align*}
\begin{aligned}
J_2 &= -\int_{\bbs^1}\bigg(\Delta\log\frac{\rho_{t}}{e^{\omega \cdot \Omega_{\rho_{t}}}}+|\nabla\log\frac{\rho_{t}}{e^{\omega \cdot \Omega_{\rho_{t}}}}|^{2}+\nabla (\omega \cdot \Omega_{\rho_{t}})\cdot\nabla\log\frac{\rho_{t}}{e^{\omega \cdot \Omega_{\rho_{t}}}}\bigg)\partial_{t}\rho_{t} \hspace{1mm}d\omega\\
&=\int_{\bbs^1}\nabla\bigg(\Delta\log\frac{\rho_{t}}{e^{\omega \cdot \Omega_{\rho_{t}}}}+|\nabla\log\frac{\rho_{t}}{e^{\omega \cdot \Omega_{\rho_{t}}}}|^{2}+\nabla (\omega \cdot \Omega_{\rho_{t}})\cdot\nabla\log\frac{\rho_{t}}{e^{\omega \cdot \Omega_{\rho_{t}}}}\bigg)\cdot\nabla\log\frac{\rho_{t}}{e^{\omega\cdot\Omega_{\rho_{t}}}}\rho_t \hspace{1mm}d\omega\\
&=\int_{\bbs^1}\nabla\Delta\log\frac{\rho_{t}}{e^{\omega \cdot \Omega_{\rho_{t}}}}\cdot\nabla\log\frac{\rho_{t}}{e^{\omega\cdot\Omega_{\rho_{t}}}}\rho_t\,d\omega
+\int_{\bbs^1}\nabla|\nabla\log\frac{\rho_{t}}{e^{\omega \cdot \Omega_{\rho_{t}}}}|^{2}\cdot\nabla\log\frac{\rho_{t}}{e^{\omega\cdot\Omega_{\rho_{t}}}}\rho_t\,d\omega\\
&\quad+\int_{\bbs^1}\nabla\Big(\nabla (\omega \cdot \Omega_{\rho_{t}})\cdot\nabla\log\frac{\rho_{t}}{e^{\omega \cdot \Omega_{\rho_{t}}}}\Big)\cdot\nabla\log\frac{\rho_{t}}{e^{\omega\cdot\Omega_{\rho_{t}}}}\rho_t\,d\omega
\end{aligned}
\end{align*}
Then, we combine $J_2$ with $J_1$ as
\begin{align*}
\begin{aligned}
&J_1+J_2\\
& =2\int_{\bbs^1}\nabla\Delta\log\frac{\rho_{t}}{e^{\omega \cdot \Omega_{\rho_{t}}}}\cdot\nabla\log\frac{\rho_{t}}{e^{\omega\cdot\Omega_{\rho_{t}}}}\rho_t\,d\omega
+\int_{\bbs^1}\nabla|\nabla\log\frac{\rho_{t}}{e^{\omega \cdot \Omega_{\rho_{t}}}}|^{2}\cdot\nabla\log\frac{\rho_{t}}{e^{\omega\cdot\Omega_{\rho_{t}}}}\rho_t\,d\omega\\
&\quad+2\int_{\bbs^1}\nabla\Big(\nabla (\omega \cdot \Omega_{\rho_{t}})\cdot\nabla\log\frac{\rho_{t}}{e^{\omega \cdot \Omega_{\rho_{t}}}}\Big)\cdot\nabla\log\frac{\rho_{t}}{e^{\omega\cdot\Omega_{\rho_{t}}}}\rho_t\,d\omega\\
& =2\int_{\bbs^1}\nabla\Delta\log\frac{\rho_{t}}{e^{\omega \cdot \Omega_{\rho_{t}}}}\cdot\nabla\log\frac{\rho_{t}}{e^{\omega\cdot\Omega_{\rho_{t}}}}\rho_t\,d\omega
+\int_{\bbs^1}\nabla|\nabla\log\frac{\rho_{t}}{e^{\omega \cdot \Omega_{\rho_{t}}}}|^{2}\cdot\Big(\nabla\rho_{t} -\nabla(\omega\cdot\Omega_{\rho_t}) \rho_t\Big)\,d\omega\\
&\quad+2\int_{\bbs^1}\nabla\Big(\nabla (\omega \cdot \Omega_{\rho_{t}})\cdot\nabla\log\frac{\rho_{t}}{e^{\omega \cdot \Omega_{\rho_{t}}}}\Big)\cdot\nabla\log\frac{\rho_{t}}{e^{\omega\cdot\Omega_{\rho_{t}}}}\rho_t\,d\omega\\
& =2\int_{\bbs^1}\Big[\nabla\Delta\log\frac{\rho_{t}}{e^{\omega \cdot \Omega_{\rho_{t}}}}\cdot\nabla\log\frac{\rho_{t}}{e^{\omega\cdot\Omega_{\rho_{t}}}}-\frac{1}{2}\Delta|\nabla\log\frac{\rho_{t}}{e^{\omega \cdot \Omega_{\rho_{t}}}}|^{2}\Big]\rho_{t}\,d\omega\\
&\quad+2\int_{\bbs^1}\Big[\nabla\Big(\nabla (\omega \cdot \Omega_{\rho_{t}})\cdot\nabla\log\frac{\rho_{t}}{e^{\omega \cdot \Omega_{\rho_{t}}}}\Big)\cdot\nabla\log\frac{\rho_{t}}{e^{\omega\cdot\Omega_{\rho_{t}}}}
-\frac{1}{2}\nabla|\nabla\log\frac{\rho_{t}}{e^{\omega \cdot \Omega_{\rho_{t}}}}|^{2}\cdot\nabla(\omega\cdot\Omega_{\rho_t})\Big]
\rho_t\,d\omega\\
&=-2\int \mbox{tr}\Big(\Big[D^{2}\log\frac{\rho_{t}}{e^{\omega\cdot\Omega_{\rho_{t}}}}\Big]^{T}D^{2}\log\frac{\rho_{t}}{e^{\omega\cdot\Omega_{\rho_{t}}}}\Big)\rho_{t} \hspace{1mm}d\omega\\
&\quad+2\int\nabla\log\frac{\rho_{t}}{e^{\omega\cdot\Omega_{\rho_{t}}}}D^{2}(\omega\cdot\Omega_{\rho_{t}})\nabla\log\frac{\rho_{t}}{e^{\omega\cdot\Omega_{\rho_{t}}}}\rho_{t} \hspace{1mm}d\omega\\
\end{aligned}
\end{align*}
where we have used the Bochner formula on $\bbs^1$ to get the first integral in the last inequality.\\
On the other hand, by integration by parts, we get
\begin{align*}
\begin{aligned}
J_3&=2\int_{\bbs^1}\Delta\rho_{t}\partial_{t}(\omega\cdot\Omega_{\rho_{t}}) \hspace{1mm}d\omega+2\int_{\bbs^1}\nabla(\omega\cdot\Omega_{\rho_{t}})\cdot\nabla\rho_{t}\partial_{t}(\omega\cdot\Omega_{\rho_{t}}) \hspace{1mm}d\omega\\
&\quad+2\int_{\bbs^1}\nabla(\omega\cdot\Omega_{\rho_{t}})\cdot\nabla\partial_{t}(\omega\cdot\Omega_{\rho_{t}})\rho_{t} \hspace{1mm}d\omega-2\int_{\bbs^1}\nabla(\omega\cdot\Omega_{\rho_{t}})\cdot\nabla\rho_{t}\partial_{t}(\omega\cdot\Omega_{\rho_{t}}) \hspace{1mm}d\omega\\
&=2\int_{\bbs^1}\Delta\rho_{t}\partial_{t}(\omega\cdot\Omega_{\rho_{t}}) \hspace{1mm}d\omega
+2\int_{\bbs^1}\nabla(\omega\cdot\Omega_{\rho_{t}})\cdot\nabla\partial_{t}(\omega\cdot\Omega_{\rho_{t}})\rho_{t} \hspace{1mm}d\omega\\
&=-2\int_{\bbs^1}\nabla\rho_{t}\cdot\nabla(\omega\cdot\partial_{t}\Omega_{\rho_{t}}) \hspace{1mm}d\omega
+2\int_{\bbs^1}\nabla(\omega\cdot\Omega_{\rho_{t}})\cdot\nabla(\omega\cdot\partial_{t}\Omega_{\rho_{t}})\rho_{t} \hspace{1mm}d\omega\\
&=-2\int_{\bbs^1}\bbp_{\omega^{\perp}}\partial_{t}\Omega_{\rho_t}\cdot\nabla\log\frac{\rho_{t}}{e^{\omega\cdot\Omega_{\rho_{t}}}}\rho_t\hspace{1mm}d\omega\\
&=-2\partial_{t}\Omega_{\rho_t}\cdot\int_{\bbs^1}\nabla\log\frac{\rho_{t}}{e^{\omega\cdot\Omega_{\rho_{t}}}}\rho_t\hspace{1mm}d\omega,
\end{aligned}
\end{align*}
where we have used \eqref{formula-1} and \eqref{formula-2}. \\
Then, we use \eqref{main-0} and \eqref{formula-1}-\eqref{formula-2} to compute $\partial_{t}\Omega_{\rho_t}$ as
\begin{align}\label{angular}
\begin{aligned}
\partial_{t}\Omega_{\rho_{t}}&=\frac{\partial_{t}J_{\rho_{t}}}{|J_{\rho_{t}}|}-\frac{J_{\rho_{t}}}{|J_{\rho_{t}}|^{2}}\frac{J_{\rho_{t}}\cdot\partial_{t}J_{\rho_{t}}}{|J_{\rho_{t}}|}=\frac{\partial_{t}J_{\rho_{t}}}{|J_{\rho_{t}}|}-\frac{\Omega_{\rho_{t}}}{|J_{\rho_{t}}|}\int_{\bbs^1}\omega\cdot\Omega_{\rho_{t}}\partial_{t}\rho_{t} \hspace{1mm}d\omega\\
&=\frac{\partial_{t}J_{\rho_{t}}}{|J_{\rho_{t}}|}+\frac{\Omega_{\rho_{t}}}{|J_{\rho_{t}}|}\int_{\bbs^1}\nabla(\omega\cdot\Omega_{\rho_{t}})\cdot\nabla\log\frac{\rho_{t}}{e^{\omega\cdot\Omega_{\rho_{t}}}}\rho_{t} \hspace{1mm}d\omega\\
&=\frac{\partial_{t}J_{\rho_{t}}}{|J_{\rho_{t}}|}+\frac{\Omega_{\rho_{t}}}{|J_{\rho_{t}}|}\Omega_{\rho_{t}}\cdot\int_{\bbs^1}\nabla\log\frac{\rho_{t}}{e^{\omega\cdot\Omega_{\rho_{t}}}}\rho_{t} \hspace{1mm}d\omega.\\ 
\end{aligned}
\end{align}
Again, using \eqref{main-0} and \eqref{formula-1}-\eqref{formula-2}, we have
\begin{align*}
\begin{aligned}
J_3&=-\frac{2}{|J_{\rho_{t}}|}\int_{\bbs^1}\Big(\omega\cdot\int_{\bbs^1}\nabla_{\omega^{\prime}}\log\frac{\rho_{t}}{e^{\omega'\cdot\Omega_{\rho_{t}}}}\rho_{t} d\omega^{\prime}\Big)\partial_{t}\rho_{t} \hspace{1mm}d\omega
-\frac{2}{|J_{\rho_{t}}|}\Big(\Omega_{\rho_{t}}\cdot\int_{\bbs^1}\nabla\log\frac{\rho_{t}}{e^{\omega\cdot\Omega_{\rho_{t}}}}\rho_{t} \hspace{1mm}d\omega\Big)^{2}\\
&=\frac{2}{|J_{\rho_{t}}|}\int_{\bbs^1}\nabla_{\omega}\Big(\omega\cdot\int_{\bbs^1}\nabla_{\omega^{\prime}}\log\frac{\rho_{t}}{e^{\omega^{\prime}\cdot\Omega_{\rho_{t}}}}\rho_{t} d\omega^{\prime}\Big)\cdot\nabla_{\omega}\log\frac{\rho_{t}}{e^{\omega\cdot\Omega_{\rho_{t}}}}\rho_{t} d\omega\\
&\quad-\frac{2}{|J_{\rho_{t}}|}\Big(\Omega_{\rho_{t}}\cdot\int_{\bbs^1}\nabla\log\frac{\rho_{t}}{e^{\omega\cdot\Omega_{\rho_{t}}}}\rho_{t} \hspace{1mm}d\omega\Big)^{2}\\
&=\frac{2}{|J_{\rho_{t}}|}\int_{\bbs^1}\bbp_{\omega^{\perp}}\int_{\bbs^1}\nabla_{\omega^{\prime}}\log\frac{\rho_{t}}{e^{\omega^{\prime}\cdot\Omega_{\rho_{t}}}}\rho_{t} d\omega^{\prime}\cdot\nabla_{\omega}\log\frac{\rho_{t}}{e^{\omega\cdot\Omega_{\rho_{t}}}}\rho_{t} d\omega\\
&\quad-\frac{2}{|J_{\rho_{t}}|}\Big(\Omega_{\rho_{t}}\cdot\int_{\bbs^1}\nabla\log\frac{\rho_{t}}{e^{\omega\cdot\Omega_{\rho_{t}}}}\rho_{t} \hspace{1mm}d\omega\Big)^{2}\\
&=\frac{2}{|J_{\rho_{t}}|}\bigg(\bigg|\int_{\bbs^1}\nabla\log\frac{\rho_{t}}{e^{\omega\cdot\Omega_{\rho_{t}}}}\rho_{t} d\omega\bigg|^{2}-\Big(\int_{\bbs^1}\Omega_{\rho_{t}}\cdot\nabla\log\frac{\rho_{t}}{e^{\omega\cdot\Omega_{\rho_{t}}}}\rho_{t} d\omega\Big)^{2}\bigg),
\end{aligned}
\end{align*}
where in the second equality, we have used  \eqref{formula-1} and \eqref{formula-2}.\\
Hence we complete \eqref{claim-0}.
\end{proof}

\subsection{Proof of \eqref{diss-small}} 
For the notational simplification, for any fixed constant $\eps_*$ satisfying $0<\varepsilon_{*}\le \frac{1}{10}$,
\[
C_{*}:= \frac{2\pi^{2}e^{2(1+|\log C_{M}|)}(1+\frac{1}{15}\varepsilon_{*})}{1-\frac{7}{6}\varepsilon_{*}}, 
\]
denotes the coefficient of the logarithmic Sobolev inequality in Lemma \ref{lem-LSI} when $\lambda=1+|\log C_{M}|$ because
of $\Psi=-\omega\cdot\Omega-\log C_M$ in our case.  Moreover, we set
\beq\label{T-star}
T_*:=2H(\rho_0 | M_{\Omega_{\rho_0}})\Big[\min\Big(C_M^2e^{-2}\eps_*^2, \frac{L}{2C_* }C_M^2e^{-2}\eps_*^2, m  C_*^{-1} \Big)\Big]^{-1},
\eeq
where $L:=(2+\frac{4}{m})^{-1}\log 2$, and $m:=\Big|\int_{\bbs^{1}}\omega M_{\Omega_{\rho}} \,d\omega\Big|.$\\
Since $H(\rho_t | M_{\Omega_{\rho_t}})\ge 0$ for all $t$ by Jensen's inequality, \eqref{ineq-1} yields that for any $t>0$,
\begin{align}
\begin{aligned}\label{cheby}
H(\rho_0 | M_{\Omega_{\rho_0}})&\ge H(\rho_0 | M_{\Omega_{\rho_0}})-H(\rho_t | M_{\Omega_{\rho_t}})=-\int_{0}^{t}\frac{d}{ds}H(\rho_s | M_{\Omega_{\rho_s}})ds\\
&=\int_{0}^{t} I(\rho_s | M_{\Omega_{\rho_s}})ds \ge t\min_{s\in[0,t]} I(\rho_s | M_{\Omega_{\rho_s}}).
\end{aligned}
\end{align}
(We remark that Lemma \ref{lem-DC} implies that the dissipation is lower semicontinuous
as a function of time and  consequently it achieves a minimum on any finite closed interval.)\\
Thus, there exists $T_0\in [0, T_*]$ such that
\beq\label{I-est}
I(\rho_{T_0} | M_{\Omega_{\rho_{T_0}}}) \le \min_{s\in[0,T_*]} I(\rho_s | M_{\Omega_{\rho_s}})\le \frac{H(\rho_0 | M_{\Omega_{\rho_0}})}{T_*}\le  \frac{1}{2} C_M^2e^{-2}\eps_*^2.
\eeq
Therefore, it follows from \eqref{Sobo} that 
\[
\|\rho_{T_0} M_{\Omega_{\rho_{T_0}}}^{-1} -1\|_{L^{\infty}(\bbs^1)}\le \eps_*,
\]
which together with Lemma \ref{lem-LSI} implies that
\[
H(\rho_{T_0} | M_{\Omega_{\rho_{T_0}}})\le C_* I(\rho_{T_0} | M_{\Omega_{\rho_{T_0}}}).
\]
Then, it follows from \eqref{T-star} and \eqref{I-est} that
\beq\label{H-est-0}
H(\rho_{T_0} | M_{\Omega_{\rho_{T_0}}})\le C_* \frac{H(\rho_0 | M_{\Omega_{\rho_0}})}{T_*}\le \frac{m}{2}.
\eeq
In particular, for all $t\ge T_0$,
\[
H(\rho_{t} | M_{\Omega_{\rho_{t}}})\le \frac{m}{2}.
\]
Then, using Csiszar-Kullback-Pinsker inequality (see for example \cite[Theorem 1.4]{entropy-inequality}), we have that for all $t\ge T_0$,
\begin{align}
\begin{aligned}\label{CKP}
|J_{\rho_t}|
&
\geq \bigg|\int_{\bbs^{1}}\omega\cdot M_{\Omega_{\rho_t}} \,d\omega\bigg|-\bigg| J_{\rho_t}-\int_{\bbs^{1}}\omega\cdot M_{\Omega_{\rho_t}} \,d\omega \bigg|
\\
&=m- \bigg|\int_{\bbs^{1}}\omega\,(\rho_t-M_{\Omega_{\rho_t}})\,d\omega\bigg| \ge m-\|\rho_t-M_{\Omega_{\rho_t}}\|_{L^{1}(\mathbb{S}^{1})}\ge \frac{m}{2},
\end{aligned}
\end{align}
which together with Lemma \ref{lem-DC} implies that for any $t\ge s\ge T_0$,
\beq\label{I-e}
I(\rho_t | M_{\Omega_{\rho_t}})\leq I(\rho_s | M_{\Omega_{\rho_s}}) e^{(2+\frac{4}{m})(t-s)}.
\eeq
Therefore, using \eqref{I-est}, we have that for any $t\in [T_0,T_0+L]$ (recall $L=(2+\frac{4}{m})^{-1}\log 2$),
\beq\label{sub-1}
I(\rho_t | M_{\Omega_{\rho_t}})\leq I(\rho_{T_0} | M_{\Omega_{\rho_{T_0}}}) e^{(2+\frac{4}{m})L} = 2 I(\rho_{T_0} | M_{\Omega_{\rho_{T_0}}})\le C_M^2e^{-2}\eps_*^2.
\eeq
Now, we extend the time length, on which the above inequality still holds. Using the same estimates as \eqref{cheby}, we have
\begin{align*}
\begin{aligned}
H(\rho_{T_0} | M_{\Omega_{\rho_{T_0}}}) &\ge H(\rho_{T_0+L/2} | M_{\Omega_{\rho_{T_0+L/2}}})\ge\int_{T_0+L/2}^{T_0+L} I(\rho_s | M_{\Omega_{\rho_s}})ds\\
& \ge \frac{L}{2}\min_{s\in[T_0+L/2,T_0+L]} I(\rho_s | M_{\Omega_{\rho_s}}).
\end{aligned}
\end{align*}
Then, thanks to \eqref{T-star} and \eqref{H-est-0}, there exists $s_*\in [T_0+L/2,T_0+L]$ such that
\[
I(\rho_{s_*} | M_{\Omega_{\rho_{s_*}}}) \le \min_{s\in[T_0+L/2,T_0+L]} I(\rho_s | M_{\Omega_{\rho_s}})\le  \frac{2}{L}\frac{C_*H(\rho_0 | M_{\Omega_{\rho_0}})}{T_*}\le \frac{1}{2\pi}C_M^2e^{-2}\eps_*^2.
\]
Thus, the inequality \eqref{I-e} yields that for any $t\in [s_*,s_*+L]$,
\[
I(\rho_t | M_{\Omega_{\rho_t}})\leq I(\rho_{s_*} | M_{\Omega_{\rho_{s_*}}}) e^{(2+\frac{4}{m})L} = 2 I(\rho_{s_*} | M_{\Omega_{\rho_{s_*}}})\le C_M^2e^{-2}\eps_*^2,
\]
which together with \eqref{sub-1} implies that 
\[
I(\rho_t | M_{\Omega_{\rho_t}}) \le C_M^2e^{-2}\eps_*^2,\quad \forall t\in [T_0,T_0+\frac{3}{2}L].
\]
We use again the above arguments to make the above inequality hold on the interval $[T_0,T_0+2L]$, therefore we repeat the same arguments to eventually complete \eqref{diss-small}.

\subsection{Proof of \eqref{remark-ineq}}
First of all, it follows from \eqref{Sobo} and \eqref{diss-small} that
\[
\|\rho M_{\Omega_{\rho}}^{-1} -1\|_{L^{\infty}(\bbs^1)}\le  \eps_*,\quad \forall t\ge T_0.
\]
Then, Lemma \ref{lem-LSI} and \eqref{ineq-1} imply that
\begin{equation}\label{sob3}
H(\rho_{t}|M_{\Omega_{\rho_{t}}})\leq C_* I(\rho_{t}|M_{\Omega_{\rho_{t}}}),\quad \forall t\ge T_0.
\end{equation}
Now, using \eqref{J-positive}
and Lemma \ref{lem-DC}, we have that for any $0<t<T_0$,
\begin{multline*}
H(\rho_{t}|M_{\Omega_{\rho_{t}}})-H(\rho_{T_{0}}|M_{\Omega_{\rho_{T_{0}}}})=-\int_{t}^{T_{0}}\frac{d}{ds}H(\rho_{s}|M_{\Omega_{\rho_{s}}})ds=\int_{t}^{T_{0}}I(\rho_{s}|M_{\Omega_{\rho_{s}}})ds\\
\leq I(\rho_{t}|M_{\Omega_{\rho_{t}}})\int_{t}^{T_{0}}{\rm \exp}\bigg[\bigg(2+\frac{2}{|J_{\rho_{0}}|e^{-T_{0}}}\bigg)s\bigg]ds \\
=I(\rho_{t}|M_{\Omega_{\rho_{t}}})\underbrace{\frac{|J_{\rho_0}|e^{-T_{0}}}{2|J_{\rho_0}|e^{-T_{0}}+2}\bigg({\rm \exp}\Big[\Big(2+\frac{2}{|J_{\rho_{0}}|e^{-T_{0}}}\bigg)T_{0}\Big]-{\rm \exp}\Big[\Big(2+\frac{2}{|J_{\rho_{0}}|e^{-T_{0}}}\Big)t\Big]\bigg)}_{=:\beta(t)}.
\end{multline*}
Thus, \eqref{sob3} and Lemma \ref{lem-DC} yield that for any $0<t<T_0$,
\begin{align*}
\begin{aligned}H(\rho_{t}|M_{\Omega_{\rho_{t}}}) & \leq I(\rho_{t}|M_{\Omega_{\rho_{t}}})\beta(t)+H(\rho_{T_{0}}|M_{\Omega_{\rho_{T_{0}}}})\\
 & \leq I(\rho_{t}|M_{\Omega_{\rho_{t}}})\beta(t)+C_*I(\rho_{T_{0}}|M_{\Omega_{\rho_{T_{0}}}})\\
 & \leq\bigg(\beta(t)+C_*\exp\bigg[\bigg(2+\frac{2}{|J_{\rho_{0}}|e^{-T_{0}}}\bigg)(T_{0}-t)\bigg]\bigg)I(\rho_{t}|M_{\Omega_{\rho_{t}}}),
\end{aligned}
\end{align*}
which together with \eqref{sob3} implies that
\[
H(\rho_{t}|M_{\Omega_{\rho_{t}}})\leq\frac{1}{B(t)}I(\rho_{t}|M_{\Omega_{\rho_{t}}}),\hspace{1em}\forall{t>0},
\]
where
\begin{multline}\label{B(t)}
B(t):=\bigg(\frac{|J_{\rho_0}|e^{-T_{0}}}{2|J_{\rho_0}|e^{-T_{0}}+2}\bigg({\rm \exp}\bigg[\bigg(2+\frac{2}{|J_{\rho_{0}}|e^{-T_{0}}}\bigg)T_{0}\bigg]-{\rm \exp}\bigg[\bigg(2+\frac{2}{|J_{\rho_{0}}|e^{-T_{0}}}\bigg)t\bigg]\bigg)_{+}\\
+C_* \exp\bigg[\bigg(2+\frac{2}{|J_{\rho_{0}}|e^{-T_{0}}}\bigg)(T_{0}-t)_{+}\bigg]\bigg)^{-1}.
\end{multline}
Therefore, it follows from \eqref{ineq-1} that
\[
\frac{d}{dt}H(\rho_{t}|M_{\Omega_{\rho_t}})\le -B(t)H(\rho_{t}|M_{\Omega_{\rho_t}}),
\]
which implies that
\[
H(\rho_{t}|M_{\Omega_{\rho_t}})\leq H(\rho_{0}|M_{\Omega_{\rho_{0}}})e^{-\int_{0}^{t} B(s)ds}.
\]
Then, the Csiszar-Kullback-Pinsker inequality yields that
\beq\label{con-1}
\|\rho_t - M_{\Omega_{\rho_t}}\|_{L^1(\bbs^1)}\le H(\rho_{0}|M_{\Omega_{\rho_{0}}})e^{-\int_{0}^{t} B(s)ds},  \quad \forall t\ge 0.
\eeq
Following the same estimates as in proof of \cite[Theorem 2.2]{F-K-M}, there exists a constant $C>0$ such that
\[
\Big|\frac{d}{dt} J_{\rho_t}\Big|\le C \|\rho_t - M_{\Omega_{\rho_t}}\|_{L^1(\bbs^{1})},
\]
therefore, 
\[
\Big|\frac{d}{dt} J_{\rho_t}\Big|\le C H(\rho_{0}|M_{\Omega_{\rho_{0}}}) e^{-\int_{0}^{t} B(s)ds}.
\]
Notice that since $e^{-\int_{0}^{t} B(s)ds}\to 0$ as $t\to\infty$, and thanks to \eqref{J-positive}, there exists a constant vector $J_{\infty}$ such that $|J_{\infty}|\ge \frac{m}{2}$ and
\begin{align*}
\begin{aligned}
|J_{\rho_t} -  J_{\infty}| &\le\int_{t}^{\infty} \bigg|\frac{d}{dr} J_{\rho_r}\bigg|dr \\
&\le C H(\rho_{0}|M_{\Omega_{\rho_{0}}})\int_{t}^{\infty}e^{-\int_{0}^{r}B(s)ds}dr.
\end{aligned}
\end{align*}
Then, setting $\Omega_{\infty}=\frac{J_{\infty}}{|J_{\infty}|}$, we have
\begin{align*}
\begin{aligned}
|\Omega_{\rho_{t}}-\Omega_{\infty}|&\le\frac{2|J_{\rho_t} -  J_{\infty}|}{|J_{\infty}|}\le \frac{4}{m}|J_{\rho_t} -  J_{\infty}|\\
&\le C\frac{4}{m}H(\rho_{0}|M_{\Omega_{\rho_{0}}})\int_{t}^{\infty}e^{-\int_{0}^{r}B(s)ds}dr.
\end{aligned}
\end{align*}
Let $\gamma_{t}\hspace{1mm}:\hspace{1mm}[0,1]\rightarrow\mathbb{R}^{2}$ be a curve defined by 
\[
\gamma_{t}(s)=(1-s)\Omega_{\rho_{t}}+s\Omega_{\infty}.
\]
Then, we have
\begin{align*}
\begin{aligned}
\|M_{\Omega_{\rho_{t}}}-
M_{\Omega_{\infty}}\|_{L^{1}(\mathbb{S}^{1})}&=C_{M}\int_{\mathbb{S}^{1}}|e^{\omega\cdot\Omega_{\rho_{t}}}
-e^{\omega\cdot\Omega_{\infty}}|d\omega\leq C_{M}\int_{\mathbb{S}^{1}}\int_{0}^{1}\bigg|\frac{d}{ds}
e^{\omega\cdot\gamma_{t}(s)}\bigg|ds \hspace{1mm}d\omega\\
&\leq C_{M}|\Omega_{\rho_{t}}
-\Omega_{\infty}|\int_{\mathbb{S}^{1}}\int_{0}^{1}
e^{\omega\cdot\gamma_{t}(s)}dsd\omega\leq2\pi e^2C_{M}|\Omega_{\rho_{t}}-\Omega_{\infty}|\\
&\le \frac{8\pi^{2}e^2CC_{M}}{m}H(\rho_{0}|M_{\Omega_{\rho_{0}}})\int_{t}^{\infty}e^{-\int_{0}^{r}B(s)ds}dr.
\end{aligned}
\end{align*}
Hence, we combine the above estimate with \eqref{con-1}, to get
\begin{align}
\begin{aligned}\label{main-ineq}
\|\rho_{t}-M_{\Omega_{\infty}}\|_{L^{1}(\mathbb{S}^{1})}
&\leq\|\rho_{t}-M_{\Omega_{\rho_{t}}}\|_{L^{1}(\mathbb{S}^{1})}
+\|M_{\Omega_{\rho_{t}}}-M_{\Omega_{\infty}}\|_{L^{1}(\mathbb{S}^{1})}\\
&\le H(\rho_{0}|M_{\Omega_{\rho_{0}}})\bigg(e^{-\int_{0}^{t} B(s)ds}+\frac{8\pi^{2}e^2CC_{M}}{m}\int_{t}^{\infty}e^{-\int_{0}^{r}B(s)ds}dr \bigg)\\
&\le \Big(\int_{\bbs^1}\rho_0\log\rho_0\,d\omega +1-\log C_{M} \Big) \\
&\quad\times \bigg(e^{-\int_{0}^{t} B(s)ds}+\frac{8\pi^{2}e^2CC_{M}}{m}\int_{t}^{\infty}e^{-\int_{0}^{r}B(s)ds}dr \bigg).
\end{aligned}
\end{align}

\subsection{Conclusion}
We now use \eqref{main-ineq} to show the three results 1)-3) in Theorem \ref{thm-converge}.\\
Since the function $B(t)$ is positive and non-decreasing, \eqref{main-ineq} implies that 
\begin{align*}
\begin{aligned}
\|\rho_{t}-M_{\Omega_{\infty}}\|_{L^{1}(\mathbb{S}^{1})}
&\le  \Big(\int_{\bbs^1}\rho_0\log\rho_0\,d\omega +1-\log C_{M} \Big) \bigg(e^{-B(0)t}+C\int_{t}^{\infty}e^{-B(0)r}dr \bigg)\\
&\le C \Big(\int_{\bbs^1}\rho_0\log\rho_0\,d\omega +1-\log C_{M} \Big)  e^{-B(0)t},
\end{aligned}
\end{align*}
where it follows from \eqref{B(t)} that
\[
B(0)=\bigg(\frac{|J_{\rho_0}|e^{-T_{0}}}{2|J_{\rho_0}|e^{-T_{0}}+2} \bigg( {\rm \exp}\bigg[\bigg(2+\frac{2}{|J_{\rho_{0}}|e^{-T_{0}}}\bigg)T_{0}\bigg] -1\bigg) + C_*{\rm \exp}\bigg[\bigg(2+\frac{2}{|J_{\rho_{0}}|e^{-T_{0}}}\bigg)T_{0}\bigg]   \bigg)^{-1}.
\] 
On the other hand, using \eqref{main-ineq} together with the fact that $B(t)=C_*^{-1}$ for all $t\ge T_0$, we have that
\[
\|\rho_{t}-M_{\Omega_{\infty}}\|_{L^{1}(\mathbb{S}^{1})} \le C \Big(\int_{\bbs^1}\rho_0\log\rho_0\,d\omega +1-\log C_{M} \Big) e^{-C_*^{-1}(t-T_0)},\quad \forall t\ge T_0.
\]
In order to show the last result, we observe that $C_*(\eps_*)\to 2\pi^{2}e^{2(1+|\log C_{M}|)}$ as $\eps_*\to 0$. Thus, for any $\eps>0$, there exists $\tilde\eps_*$ such that 
\[
C_*^{-1}(\tilde\eps_*)\ge \frac{1}{2\pi^{2}e^{2(1+|\log C_{M}|)}}-\frac{\eps}{2}.
\]
For such $\tilde\eps_*$, since 
\[
T_0 \le 2H(\rho_0 | M_{\Omega_{\rho_0}})\Big[\min\Big(C_M^2e^{-2}\tilde \eps_*^2, \frac{L}{2C_*(\tilde\eps_*) }C_M^2e^{-2}\tilde\eps_*^2, m  C_*^{-1}(\tilde\eps_*) \Big)\Big]^{-1},
\]
if $H(\rho_0 |M_{\Omega_{\rho_0}})\to 0$, then $T_0\to 0$ and $ |J_{\rho_0}| \to m$, consequently $B(0)\to C_*^{-1}(\tilde\eps_*) $.  \\
Thus, there exists $\delta>0$ such that if $H(\rho_0 |M_{\Omega_{\rho_0}})<\delta$, then
\[
B(0)\ge C_*^{-1}(\tilde\eps_*)-\frac{\eps}{2}.
\]
Therefore, we have
\begin{align*}
\begin{aligned}
\|\rho_{t}-M_{\Omega_{\infty}}\|_{L^{1}(\mathbb{S}^{1})}
&\le C \Big(\int_{\bbs^1}\rho_0\log\rho_0\,d\omega +1-\log C_{M} \Big)  e^{-B(0)t}\\
&\le C \Big(\int_{\bbs^1}\rho_0\log\rho_0\,d\omega +1-\log C_{M} \Big) \exp{\Big[-\Big(\frac{1}{2\pi^{2}e^{2(1+|\log C_{M}|)}}-\eps\Big)t\Big]}.
\end{aligned}
\end{align*}
Hence we complete the proof.

\appendix

\begin{section}{Laplacian log identity} 
\setcounter{equation}{0}
\begin{lemma} Assume $\rho$ and  $\Psi$ are functions in $\mathcal{C}^{2}(\mathbb{S}^{1})$, and $\rho>0$. Then, the following identity holds:
\begin{equation}\label{laplog}
\frac{\Delta\rho}{\rho}+\frac{\nabla \Psi\cdot\nabla\rho}{\rho}+\Delta \Psi=\Delta\log\frac{\rho}{e^{-\Psi}}+\Big|\nabla\log\frac{\rho}{e^{-\Psi}}\Big|^{2}-\nabla \Psi\cdot\nabla\log\frac{\rho}{e^{-\Psi}},
\end{equation}
\begin{proof}
We begin by computing the first term in r.h.s of \eqref{laplog} as
\begin{align*}
\begin{aligned}
\Delta\log\frac{\rho}{e^{-\Psi}}=\mbox{div}\frac{\nabla\frac{\rho}{e^{-\Psi}}}{\frac{\rho}{e^{-\Psi}}}=-\frac{\bigg|\nabla\frac{\rho}{e^{-\Psi}}\bigg|^{2}}{\bigg(\frac{\rho}{e^{-\Psi}}\bigg)^{2}}+\frac{\Delta\frac{\rho}{e^{-\Psi}}}{\frac{\rho}{e^{-\Psi}}}=:I_1+I_2. 
\end{aligned} 
\end{align*}
By straightforward computations, we have 
\[
I_1=-\Big|\nabla\log\frac{\rho}{e^{-\Psi}}\Big|^{2},
\]
and
\begin{align*}
\begin{aligned}
I_2&=\frac{\mbox{div}\bigg(\frac{\nabla\rho}{e^{-\Psi}}+\frac{\rho}{e^{-\Psi}}\nabla \Psi\bigg)}{\frac{\rho}{e^{-\Psi}}}\\
&=\frac{\Delta\rho}{\rho}+2\frac{\nabla \Psi\cdot\nabla\rho}{\rho}+\Delta \Psi+|\nabla \Psi|^{2}\\
&=\frac{\Delta\rho}{\rho}+\frac{\nabla \Psi\cdot\nabla\rho}{\rho}+\Delta \Psi+\nabla \Psi\cdot\nabla\log\frac{\rho}{e^{-\Psi}}.
\end{aligned} 
\end{align*}
Hence, we have the desired identity.
\end{proof}
\end{lemma}
\end{section}

\begin{section}{Uniqueness} 
\setcounter{equation}{0}
We here present another proof for the uniqueness of weak solutions to \eqref{main}, combining the short-time stability \eqref{stab} with the energy estimates to be specified below. First of all, we see that the weak solution $\rho$ to \eqref{main} becomes smooth instantaneously thanks to the parabolic regularity, therefore we have that $\rho\in L^{\infty}([t_0,\infty)\times \bbs^1)$ for any fixed $t_0>0$.\\
Let $\rho$ and $\bar\rho$ be any two weak solutions to \eqref{main}. Then, since 
\[
\partial_{t}(\rho-\bar\rho)= \Delta_{\omega}(\rho-\bar\rho)-\nabla_{\omega}\cdot\Big((\rho-\bar\rho)\, \bbp_{\omega^{\perp}}\Omega_{\rho}\Big)-\nabla_{\omega}\cdot\Big(\bar\rho\, \bbp_{\omega^{\perp}}(\Omega_{\rho}-\Omega_{\bar\rho})\Big),
\]
we have
\begin{align*}
\begin{aligned}
\frac{d}{dt}\int_{\bbs^1} \frac{|\rho-\bar\rho|^2}{2} d\omega+\int_{\bbs^1} |\nabla_{\omega}(\rho-\bar\rho)|^2 d\omega 
&=  \int_{\bbs^1} (\rho-\bar\rho)\Omega_{\rho} \cdot \nabla_{\omega}(\rho-\bar\rho) d\omega \\
&\quad +  \int_{\bbs^1} \bar\rho (\Omega_{\rho}-\Omega_{\bar\rho}) \cdot \nabla_{\omega}(\rho-\bar\rho) d\omega.
\end{aligned} 
\end{align*}
Since \eqref{J-positive} yields that
\[
|\Omega_{\rho}-\Omega_{\bar\rho}|\le \frac{2|J_{\rho}-J_{\bar\rho}|}{|J_{\rho}|}\le Ce^t |J_{\rho_0}|^{-1} \|\rho-\bar\rho\|_{L^2(\bbs^1)},
\]
we have
\[
\bigg|\int_{\bbs^1} \bar\rho (\Omega_{\rho}-\Omega_{\bar\rho}) \cdot \nabla_{\omega}(\rho-\bar\rho) d\omega \bigg| \le Ce^t |J_{\rho_0}|^{-1} \|\bar\rho\|_{L^{\infty}(\bbs^1)}\|\rho-\bar\rho\|_{L^2(\bbs^1)} \|\nabla_{\omega}(\rho-\bar\rho)\|_{L^2(\bbs^1)}. 
\]
Then, using $\rho\in L^{\infty}([t_0,\infty)\times \bbs^1)$ for all $t\ge t_0$, we have that for all $t\ge t_0$
\begin{align*}
\begin{aligned}
\frac{d}{dt}\int_{\bbs^1} |\rho-\bar\rho|^2 d\omega+\int_{\bbs^1} |\nabla_{\omega}(\rho-\bar\rho)|^2 d\omega 
&\le C(1+e^{2t}) \int_{\bbs^1} |\rho-\bar\rho|^2 d\omega,
\end{aligned} 
\end{align*}
which implies that 
\[
\int_{\bbs^1} |\rho-\bar\rho|^2 d\omega \le Ce^{e^{2t}} \int_{\bbs^1} |\rho_{t_0}-\bar\rho_{t_0}|^2 d\omega,\quad \forall t\ge t_0.
\]
Therefore, if $\rho_{t_0}=\bar\rho_{t_0}$, then $\rho_t=\bar\rho_t$ for all $t\ge t_0$. \\
Since it follows from the short-time stability \eqref{stab} that if $\rho_{0}=\bar\rho_{0}$, then $\rho_t=\bar\rho_t$ for all $t\le \delta$,
we take $t_0<\delta$ to complete the uniqueness.
\end{section}

\begin{section}{Formulas for Calculus on sphere} 
\setcounter{equation}{0}
We here present some useful formulas on $n$-dimensional sphere $\bbs^{n}$, which are extensively used in this paper.\\
Let $F$ be a vector-valued function and $f$ be scalar-valued function. Then we have the following formulas related to the integration by parts:
\beq\label{formula-0}
\int_{\bbs^{n}} f\nabla_{\omega}\cdot F d\omega =  -\int_{\bbs^{n}} F\cdot(\nabla_{\omega}f -2\omega f) d\omega.
\eeq
By the definition of the projection $\bbp_{\omega^{\perp}}$, it is obvious that
\beq\label{formula-1}
\bbp_{\omega^{\perp}}\nabla_{\omega} f =\nabla_{\omega} f 
\eeq
for any scalar-valued function $f$.\\
Moreover, for any constant vector $v\in\bbr^{n+1}$, we have
\begin{align}
\begin{aligned}\label{formula-2}
\nabla_{\omega}(\omega\cdot v) = \bbp_{\omega^{\perp}} v.
\end{aligned}
\end{align}
 \end{section}

\bibliographystyle{amsplain}
\bibliography{Vicsek-dim1-submit}

\end{document}